\documentclass[reqno,11pt]{amsart}

\usepackage{mathdots}
\usepackage{tikz}
\usepackage{tikz-cd}
\usetikzlibrary{automata}
\usepackage{amssymb}
\usepackage{amsgen}
\usepackage{amsmath}
\usepackage{amsthm}
\usepackage{mathrsfs}
\usepackage{cite}
\usepackage{amsfonts}
\usepackage{enumitem}

\newcommand{\rad}{\mathop{\mathrm{rad}}\nolimits}

\newcommand{\inv}{^{-1}}

\newcommand{\ov}[1]{\ensuremath{\overline {#1}}}
\newcommand{\til}[1]{\ensuremath{\widetilde {#1}}}

\newcommand{\wh}{\widehat}

\newcommand{\Tor}{\mathop{\mathrm{Tor}}\nolimits}

\usepackage{xcolor}

\newtheorem{Thm}{Theorem}[section]
\newtheorem{Prop}[Thm]{Proposition}

\newtheorem{Lemma}[Thm]{Lemma}
{\theoremstyle{definition}
}
{\theoremstyle{remark}
\newtheorem{Rmk}[Thm]{Remark}}
\newtheorem{Cor}[Thm]{Corollary}
{\theoremstyle{remark}
}
{\theoremstyle{remark}
\newtheorem{Example}[Thm]{Example}}

{\theoremstyle{remark}
}
{\theoremstyle{remark}
}
{\theoremstyle{remark}
}
{\theoremstyle{remark}
\newtheorem*{Claim*}{Claim}}
\newtheorem*{theorema}{Theorem A}

\numberwithin{equation}{section}

\title[Quiver presentations of band algebras]{Quiver presentations for band algebras are defined over the integers}

\author{Benjamin Steinberg}
\address[B.~Steinberg]{%
    Department of Mathematics\\
    City College of New York\\
    Convent Avenue at 138th Street\\
    New York, New York 10031\\
    USA}
\email{bsteinberg@ccny.cuny.edu}

\thanks{The author was supported by the NSF grant DMS-2452324, a Simons Foundation Collaboration Grant, award number 849561, the Australian Research Council Grant DP230103184, and Marsden Fund Grant MFP-VUW2411.}
\date{\today}


\begin{document}

\begin{abstract}
A band is a semigroup in which each element is idempotent.
In recent years, there has been a lot of activity on the representation theory of the subclass of left regular bands due to connections to Markov chains associated to hyperplane arrangements, oriented matroids, matroids and CAT(0) cube complexes.  We prove here that the integral semigroup algebra of a band is isomorphic to the integral path algebra of a quiver modulo an admissible ideal.  This leads to a uniform bound quiver presentation for band algebras over all fields.  Also, we answer a question of Margolis, Saliola and Steinberg by proving that the integral semigroup algebra of a CW left regular band is isomorphic to the quotient of the integral path algebra of the Hasse diagram of its support semilattice modulo the ideal generated by the sum of all paths of length two.  This includes, for example, hyperplane face semigroup algebras.
\end{abstract}
\maketitle

\section{Introduction}
Bidigare, Hanlon and Rockmore~\cite{BHR} observed that a number of important Markov chains, including the Tsetlin library and the riffle shuffle, could be modeled as random walks on face semigroups of hyperplane arrangements.  These turn out to belong to a class of semigroups known as left regular bands.  Moreover, the representation theory of left regular bands can be used to compute the eigenvalues of the transition matrices.  Brown and Diaconis~\cite{DiaconisBrown1} showed that hyperplane walks have diagonalizable transition matrices, and they gave bounds on the mixing time in terms of the eigenvalues.   Brown then generalized the theory to arbitrary left regular bands in~\cite{Brown1}, including those associated to matroids, and eventually to arbitrary bands~\cite{Brown2}.  A band is a semigroup in which each element is idempotent.

This work led to a systematic study of the representation theory of left regular band algebras in their own right.  Brown's student, Saliola, computed the quiver of a left regular band algebra~\cite{Saliola} and gave a bound quiver presentation for the algebra of the face monoid of a central hyperplane arrangement~\cite{Saliolahyperplane}. 
Margolis and the author then computed the quiver for the algebra of an arbitrary band~\cite{DO}.

Definitive results on the representation theory of left regular bands appeared in~\cite{MSS,ourmemoirs}.  Moreover,~\cite{ourmemoirs} introduced the class of CW left regular bands.  These are, roughly speaking, left regular bands whose poset of principal right ideals is the face poset of a regular CW complex.  These include hyperplane face semigroups as well as left regular bands associated to oriented matroids, complex hyperplane arrangements, CAT(0) cube complexes and complexes of oriented matroids (COMS); see~\cite{ourmemoirs} for details and references.  In particular, in~\cite{ourmemoirs} an explicit, field-independent, bound quiver presentation of the algebra of a CW left regular band was given, and it was asked whether this quiver presentation was already valid over the integers.  In this paper we answer that question positively.  

A band is called connected if its integral semigroup algebra has a unit.  We present later a combinatorial characterization of this property that explains the name.  Our main theorem is then:
\begin{theorema}
    Let $B$ be a connected band.  Then there is an explicitly defined quiver $Q(B)$ such that $\mathbb ZB\cong \mathbb ZQ(B)/I$ with $I$ an admissible ideal.  For any field $k$, $k\otimes I$ is an admissible ideal of $kQ(B)$ and $kB\cong kQ(B)/(k\otimes I)$.
\end{theorema}

Theorem~A can be interpreted as saying that the representation theory of $B$ is independent of the field.  

\section{Bands}
A \emph{band} $B$ is a semigroup in which each element is idempotent. For example, meet semilattices are precisely the commutative 
bands. All bands are assumed finite here.  
A nice introduction to bands can be found in the appendices of Brown~\cite{Brown2}, although he works with the dual order to us on the support semilattice.  

%
  Bands enjoy the following crucial property.
   \begin{Thm}[Clifford~{\cite{Clifford}}]\label{t:clifford}
   Let $B$ be a band.  Then $BaB\cap BbB=BabB$ for all $a,b\in B$.
   \end{Thm}
   \begin{proof}
   First we observe that $BxyB=ByxB$ for all $x,y\in B$ as $xy= x(yx)y$ and $yx=y(xy)x$.
   
   Trivially $BabB\subseteq BaB\cap BbB$.  Let $c\in BaB\cap BbB$ and write $c=uav=xby$.   Then $cuac=cuauav=cuav=c$ and, similarly, $cbyc=c$.  Then $c(byccua)c=c$ and so $BcB=B(byc)(cua)B=B(cua)(byc)B\subseteq BabB$ by our observation.  This completes the proof.  
   \end{proof}
   Therefore, the set $\Lambda(B)$ of principal (two-sided) ideals of $B$ is a meet semilattice, called the \emph{support semilattice} of $B$, and $\sigma\colon B\to \Lambda(B)$ given by $\sigma(b) = BbB$ is a homomorphism, called the \emph{support homomorphism}.

Two crucial properties of bands are what Brown calls the ``swallowing'' and ``deletion'' properties~\cite{Brown2}.

\begin{Prop}\label{p:deletion}
Let $B$ be a band.
\begin{enumerate}
    \item  If $a\in BbB$, then $aba=a$. \hfill\emph{(Swallowing)}
    \item  If $BaB=BcB\subseteq BbB$, then $abc=ac$. \hfill\emph{(Deletion)}
\end{enumerate}
\end{Prop}
\begin{proof}
    For swallowing, if  $a=xby$, then $a=axbya$.  Therefore, $axba=axb(axbya)=axbya=a$ and, hence, $aba=axba(ba)=axba=a$.
    Deletion follows from Theorem~\ref{t:clifford} and two applications of swallowing: $abc = ab(cac)=(abca)c=ac$.
\end{proof}


A band is called \emph{left regular} if it also satisfies $xyx=xy$.  In this case, $BaB=Ba$, and so the semilattice of principal ideals is the semilattice of principal left ideals.  Most of the interest in bands comes from the case of left regular bands. See~\cite{Brown1} and~\cite[Chapter~2]{ourmemoirs} for more on left regular bands.

%
%
There are natural  preorders $\leq_{\mathscr L}$ and $\leq_{\mathscr R}$, due to Green~\cite{Green}, that can be defined on any semigroup.   For a band $B$ they take the form $a\leq_{\mathscr R} b$ if and only if $aB\subseteq bB$, if and only if $ba=a$.  Dually, $a\leq_{\mathscr L} b$ if and only if $Ba\subseteq Bb$, if and only if $ab=a$.    On a left regular band, $\leq_{\mathscr R}$ is a partial order. 
If $\leq$ is a preorder on a set, we write $p<q$ if $p\leq q$, but $q\nleq p$.
It is important to note that $\sigma$ preserves strict inequalities. Indeed, it clearly preserves $\leq_{\mathscr R}$, being a homomorphism. If $a\leq _{\mathscr R} b$ and $\sigma(a)=\sigma(b)$, then $b=bab=ab$, and so $b\leq_{\mathscr R} a$. A dual argument applies to $<_{\mathscr L}$.

\begin{Rmk}\label{rmk:useful}
    We shall frequently use that if $\sigma(a)=\sigma(b) = \sigma(c)$ and $a,b\in Bc$, then $ba=bac = bc=b$ by deletion.
\end{Rmk}

If $X\in \Lambda(B)$, then $B_{\geq X}=\{b\in B\mid \sigma(b)\geq X\}$ is a subsemigroup of $B$ with support semilattice $\Lambda(B)_{\geq X}$.   We associate two simple graphs to each $X\in \Lambda(B)$. The graph $\Gamma_{\mathscr R, B}(X)$ has vertex set $\sigma\inv(X)$ and we define $x,x'$ to be adjacent if they have a common upper bound in the preorder $\leq_{\mathscr R}$.  The graph $\Gamma_{\mathscr L, B}(X)$ is defined dually using $\leq_{\mathscr L}$.  We say that $B$ is \emph{connected} if $\Gamma_{\mathscr L, B}(X)$ and $\Gamma_{\mathscr R, B}(X)$ are connected for each $X\in \Lambda(B)$.  Note that $\Gamma_{\mathscr L, B}(X)$ is automatically connected for a left regular band as $\sigma(x)=X=\sigma(x')$ implies that $x=xx$, $x'=x'x$, and so $x$ is a common $\leq_{\mathscr L}$-upper bound.  We shall see that connectivity of a band is equivalent to its semigroup algebra having a unit.

If $P$ is a preordered set, then $P^{\leq p} = \{q\in P\mid q\leq p\}$ and $P^{<q} = \{q\in P\mid q< p\}$.
If $x\in B$, then $B^{\leq_\mathscr R x}=xB$ and 
$B^{<_\mathscr R x}$ are right ideals of  $B$. The support semilattice of $B^{<_{\mathscr R}x}$ can be identified with $\Lambda(B)^{<X}$.
Dual remarks hold for $B^{\leq_\mathscr L x}=Bx$ and $B^{<_\mathscr L x}$.

\section{Band algebras}
Fix a commutative ring $k$ and band $B$.  If $B$ is left regular, then the semigroup algebra $kB$ is unital if and only if $B$ is connected~\cite[Theorem~4.15]{ourmemoirs}.  We extend this result to the general case here.  We shall use throughout the well-known fact that idempotents lift modulo a nilpotent ideal for unital rings~\cite[Theorem~21.28]{LamBook}. As usual, a ring is unital if it has an identity.


\begin{Lemma}\label{l:equiv}
Let $X$ be a set and $k$ a commutative ring.  Let $\equiv$ be an equivalence relation on $X$ and  $R\subseteq {\equiv}$ a binary relation.  Let $V$ be the submodule of $kX$ spanned by all differences $x-y$ with $x\equiv y$.  Then $R$ generates $\equiv$ if and only if $V$ is spanned by all differences $x-y$ with $(x,y)\in R$.  
\end{Lemma}
\begin{proof}
Let $W$ be the span of all differences $x-y$ with $(x,y)\in R$, and note that $W\subseteq V$.  
Let $\sim$ be the equivalence relation generated by $R$.    Then if $x\sim y$, we can find $x=x_0,x_1,\ldots, x_n=y$ with $(x_i,x_{i+1})$ or $(x_{i+1},x_i)$ in  $R$ for $0\leq i\leq n-1$.  Then \[x-y = (x_0-x_1)+(x_1-x_2)+\cdots +(x_{n-1}-x_n)\in W.\]   In particular, if $R$ generates $\equiv$, then $V=W$.  

Suppose, conversely, that $V=W$.  Let $\pi\colon KX\to K[X/{\sim}]$ be the map induced by the projection $X\to X/{\sim}$.  Then $\ker \pi$ is spanned by all differences $x-y$ with $x\sim y$, and hence coincides with $W$ by the above.  Thus, if $x\equiv y$, then $x-y\in V=W$, and so $\pi(x-y)=0$.  Therefore, $\pi(x)=\pi(y)$, that is, $x\sim y$. 
\end{proof}

By a celebrated result of Solomon~\cite{Burnsidealgebra}, if $L$ is a meet semilattice, then $kL\cong k^L$, the isomorphism sending $\ell\in L$ to the characteristic function $\delta_{L^{\leq \ell}}$.  Let $\sigma\colon kB\to k\Lambda(B)$ be the linear extension of $\sigma$ and let $\tau\colon kB\to k^{\Lambda(B)}$ be the composition of $\sigma$ with the Solomon isomorphism: $\tau(b)=\delta_{\Lambda(B)^{\leq \sigma(b)}}$.  Note that $\ker \sigma=\ker \tau$.

A crucial fact for us is the following theorem, proved by Brown~\cite{Brown2} over a field, and without explicitly stating the nilpotency bound.  The general case follows from the case of $\mathbb Q$.  Indeed, since the sequence $0\to \ker \sigma\to \mathbb ZB\to \mathbb Z\Lambda(B)\to 0$ splits over $\mathbb Z$, it follows that $k\otimes \ker \sigma$
is the kernel of $kB\to k\Lambda(B)$. Here all tensor products are over $\mathbb Z$.  The nilpotency of $\mathbb Q\otimes \ker\sigma$ implies the nilpotency of $\ker \sigma$ (since $\ker \sigma$ is free abelian, and hence embeds in $\mathbb Q\otimes \ker\sigma$).  This in turn yields the nilpotency of $k\otimes \ker\sigma$ for any commutative ring $k$.
 
\begin{Thm}[Brown~{\cite[Theorem~B1]{Brown2}}]\label{t:brown.nil}
Let $k$ be a commutative ring and $\sigma\colon kB\to k\Lambda(B)$ be induced by the support map.  If $\mathcal J=\ker \sigma$, then $\mathcal J^{2^{n+1}+1}=0$ where $n$ is the length of the longest chain in $\Lambda(B)$.     
\end{Thm}

In particular, when $k$ is a field, $\ker\sigma=\rad(kB)$ as $k\Lambda(B)\cong k^{\Lambda(B)}$ is semisimple.

\begin{Prop}\label{p:left.id}
Let $S$ be a finite semigroup and $k$ a commutative ring. Let $J$ be a nilpotent two-sided ideal of $kS$ such that $kS/J$ has a left identity.   Then $kS$ has a left identity if and only  if $J=kS\cdot J$.
\end{Prop}
\begin{proof}
Trivially, if $kS$ has a left identity, then $kS\cdot J=J$.  Assume now that $kS\cdot J=J$. Let $M$ be the monoid obtained by adjoining an identity $1$ to $S$, and view $kS$ as an ideal in $kM$.  Note that $J$ is also an ideal of $kM$. 
 Since $J$ is nilpotent, we can lift a left identity of $kS/J$ to an idempotent $e\in KM$, which must necessarily belong to $kS$ as $J\subseteq kS$ are ideals in $kM$.    We want to show that $(1-e)kS=0$.  By choice of $e$, we have $(1-e)(kS/J)=0$, i.e., $(1-e)kS\subseteq J$, and so $(1-e)kS=(1-e)J$.  Multiplying both sides on the right by $J$ and using that $kS\cdot J=J$, we conclude that $(1-e)J=(1-e)J^2$.  By induction, $(1-e)J=(1-e)J^n$ for all $n\geq 1$.    Since $J$ is nilpotent, it follows that $0=(1-e)J=(1-e)kS$, as required.
\end{proof}

Of course, the dual to Proposition~\ref{p:left.id} for right identities holds.  We shall apply the proposition and its dual to $\ker\sigma$ since $k\Lambda(B)\cong k^{\Lambda(B)}$ is unital.

\begin{Thm}\label{t:right.connected}
Let $B$ be a band.  Then the following are equivalent.
\begin{enumerate}
    \item $B$ is connected.
    \item $\mathbb ZB$ is unital.
    \item $kB$ is unital for all commutative rings $k$.
\end{enumerate}
\end{Thm}
\begin{proof}
Trivially, (3) implies (2).  Suppose that $\mathbb ZB$ has an identity.   We prove that all the graphs $\Gamma_{\mathscr R, B}(X)$ are connected.  The proof for the $\Gamma_{\mathscr L, B}(X)$ is dual. Letting $\mathcal J=\ker \sigma$, we trivially have that $\mathbb ZB\cdot \mathcal J = \mathcal J$. 
Thus $\mathcal J$ is generated by all differences $c(a-b)=ca-cb$ with $a,b,c\in B$ and $\sigma(a)=\sigma(b)$.  Note that $\sigma(ca)=\sigma(cb)$ and $ca,cb\leq_{\mathscr R}c$.  Thus $\mathcal J=\ker \sigma$ is spanned by differences $x-x'$ with $\sigma(x)=\sigma(x')$ and $x,x'\leq_{\mathscr R} y$ for some $y\in B$.  It follows from Lemma~\ref{l:equiv} that the equivalence relation induced by $\sigma$ is generated by all pairs $(x,x')$ such that $\sigma(x)=\sigma(x')$ and $x,x'$ are connected by an edge in $\Gamma_{\mathscr R,B}(\sigma(x))$.  Therefore, if $\sigma(a)=\sigma(b)$, then there is a path from $a$ to $b$ in $\Gamma_{\mathscr R,B}(\sigma(a))$, and hence these graphs are all connected.

Assume (1).  We show that $kB$ has a left identity.  A dual argument  establishes the existence of a right identity, and hence of a two-sided identity.  By Proposition~\ref{p:left.id}, it suffices to show that $\mathcal J=kB\cdot \mathcal J$ 
as $\mathcal J$ is nilpotent and $kB/\mathcal J\cong k\Lambda(B)\cong k^{\Lambda(B)}$ has an identity.  By our connectivity assumption, the equivalence relation induced by $\sigma$ is generated by all pairs $(x,x')$ such that $\sigma(x)=\sigma(x')$ and $x,x'$ are adjacent in $\Gamma_{\mathscr R,B}(\sigma(x))$.  Therefore, by Lemma~\ref{l:equiv}, $\mathcal J$ is spanned by all differences $x-x'$ such that $yx=x$, $yx'=x'$ for some $y\in B$ and $\sigma(x)=\sigma(x')$.  Then $x-x'=y(x-x')\in kB\cdot \mathcal J$.  We conclude that $\mathcal J=kB\cdot \mathcal J$, and hence $kB$ has a left identity.  This completes the proof.
\end{proof}

From here on out, we assume that $B$ is connected.  
Let us fix once and for all $f_X\in \sigma\inv(X)$ for each $X\in \Lambda(B)$.  Since $\mathbb ZB$ is unital and $\mathcal J=\ker \sigma=\ker\tau $ is nilpotent, we can find a complete set of orthogonal idempotents $\{e_X\mid X\in \Lambda(B)\}$ such that $\tau(e_X)=\delta_X$~\cite[Proposition~21.25]{LamBook}. 

The left \emph{Sch\"utzenberger representation} associated to $X\in \Lambda(B)$ is the $\mathbb ZB$-module $\mathbb ZB^{\leq_{\mathscr L} f_X}/\mathbb ZB^{<_{\mathscr L}f_X}$.  If $L_X=\sigma\inv(X)\cap Bf_X$, then we can identify the Sch\"utzenberger representation with the free abelian group $\mathbb ZL_X$ with basis $L_X$ and module action defined for $a\in B$ and $b\in L_X$ by
\[a\cdot b = \begin{cases}
    ab, & \text{if}\ \sigma(a)\geq X,\\ 0, & \text{else,}
\end{cases}\]
and we shall do so from now on.   The right Sch\"utzenberger representation associated to $X$ is defined dually. 

\begin{Prop}\label{p:linearindep.schutz}
    There is a $\mathbb ZB$-module homomorphism $\alpha\colon \mathbb ZBe_X\to \mathbb ZL_X$ with $\alpha(be_X) = b$ for $b\in L_X$.  
\end{Prop}
\begin{proof}
    There is a $\mathbb ZB$-linear map $\gamma\colon \mathbb ZL_X\to \mathbb Z^{\Lambda(B)}\delta_X=\mathbb Z\delta_X$ sending $b\in L_X$ to $\delta_X$.   On the other hand, $\tau$ induces a surjection $\mathbb ZBe_X\to \mathbb Z^{\Lambda(B)}\delta_X$ sending $e_X$ to $\delta_X$.  Since $\mathbb ZBe_X$ is projective, we can find $\alpha\colon \mathbb ZBe_X\to \mathbb ZL_X$ with $\tau=\gamma\alpha$ .  Putting $\ell_X=\alpha(e_X)$, we must have $\gamma(\ell_X) = \delta_X$.  Therefore, $\ell_X = \sum_{a\in L_X}c_aa$ where $\sum_{a\in L_X}c_a=1$.  Now $ba=b$ whenever $\sigma(a)=\sigma(b)=X$ and $a,b\in Bf_X$ by Remark~\ref{rmk:useful}, and so $\alpha(be_X) = b\ell_X =\sum_{a\in L_X}c_aba= \sum_{a\in L_X}c_ab = b$.  The result follows.       
\end{proof}



%
%
%
%

\section{Quivers, path algebras and admissible ideals}
A \emph{quiver} $Q$ is a finite directed graph, with multiple edges and loops allowed.  The set of vertices is denoted $Q_0$ and the set of edges $Q_1$. If $v,w$ are vertices, then $Q_1(v,w)$ will denote the set of edges from $v$ to $w$.  

If $Q$ is a quiver and $k$ is a commutative ring with unit, the \emph{path algebra} $kQ$ has basis consisting of all (directed) paths in $Q$, including an empty path $\varepsilon_v$ at each vertex $v$.  We compose paths from right to left, so that if we write $s(e)$ for the source of an edge and $t(e)$ for the target, then a path has the form $p=e_n\cdots e_2e_1$ where $t(e_i) = s(e_{i+1})$.  We put $s(p) = s(e_1)$ and $t(p)=t(e_n)$.  The multiplication in $kQ$ is then given by taking $pq$ to be the concatenation when $s(p)=t(q)$, and $pq=0$ otherwise. The identity of $kQ$ is the sum of the empty paths.

We define the arrow ideal $J_k$ to be the ideal of $kQ$ with basis all paths of positive length.  When $k=\mathbb Z$ we just write $J$.   Note that the exact sequence
\[0\to J\to \mathbb ZQ\to \mathbb ZQ_0\to 0\] splits over $\mathbb Z$ and hence
\[0\to k\otimes J\to kQ\to kQ_0\to 0\] is exact, whence $k\otimes J= J_k$.  

Let us say that an ideal $I$ of $kQ$ is \emph{admissible} if there exists $n\geq 2$ with $J_k^n\subseteq I\subseteq J_k^2$ and $kQ/I$ is projective over $k$.   This coincides with the standard definition when $k$ is a field~\cite{assem,benson}.  A pair $(Q,I)$ consisting of a quiver $Q$ and an admissible ideal $I$ is called a \emph{bound quiver} presenting the algebra $kQ/I$.

Suppose that $I$ is an admissible ideal of $\mathbb ZQ$.  Then \[0\to I\to \mathbb ZQ\to \mathbb ZQ/I\to 0\] is exact and $\mathbb ZQ/I$ is free abelian, so the sequence splits over $\mathbb Z$.  It follows 
\begin{equation*}\label{eq:tensor.over.k}
0\to k\otimes I\to kQ\to k\otimes \mathbb ZQ/I\to 0
\end{equation*}
is exact and $k\otimes \mathbb ZQ/I$ is a free $k$-module.  Moreover, $J_k^m = (k\otimes J)^m$. 
Therefore, if $J^n\subseteq I\subseteq J^2$, then $J_k^n\subseteq k\otimes I\subseteq J_k^2$, and so $k\otimes I$ is admissible.  We summarize this discussion in the following proposition.

\begin{Prop}\label{p:admiss.ov.z}
    Let $I$ be an admissible ideal of $\mathbb ZQ$.  Then for any commutative ring $k$, one has that $k\otimes I$ is an admissible ideal of $kQ$ and $kQ/(k\otimes I)\cong k\otimes \mathbb ZQ/I$.  
\end{Prop}

The same argument shows that admissibility of an ideal is preserved by base change, in general.  This is one advantage of requiring $kQ/I$ to be $k$-projective. The next folklore theorem, which generalizes the well-known results of Gabriel~\cite{assem} and Bongartz~\cite{Bongartz} over fields, gives another. 

\begin{Thm}\label{homological.stuff}
    Let $k$ be a commutative ring, $Q$ a quiver and $I$ an admissible ideal. Then $(kQ/J_k)\varepsilon_v$ and $\varepsilon_w(kQ/J_k)$ are left and right, respectively,  $kQ/I$-modules for $v,w\in Q_0$ and the following formulas hold:
    \begin{align*}
        \Tor_1^{kQ/I}(\varepsilon_w (kQ/J_k),(kQ/J_k)\varepsilon_v) &\cong \varepsilon_w(J_k/J_k^2)\varepsilon v\\
        \Tor_2^{kQ/I}(\varepsilon_w(kQ/J_k),(kQ/J_k)\varepsilon_v) &\cong \varepsilon_w(I/(IJ_k+J_kI))\varepsilon_v\\
        \Tor_{2n+1}^{kQ/I}(\varepsilon_w (kQ/J_k),(kQ/J_k)\varepsilon_v) &\cong \varepsilon_w\left(\frac{I^nJ_k\cap J_kI^n}{I^{n+1}+J_kI^nJ_k}\right)\varepsilon_v \tag{$n\geq 1$}\\
        \Tor_{2n}^{kQ/I}(\varepsilon_w (kQ/J_k),(kQ/J_k)\varepsilon_v) &\cong \varepsilon_w\left(\frac{I^n\cap J_kI^{n-1}J_k}{I^nJ_k+J_kI^n}\right)\varepsilon_v\tag{$n\geq 2$}
    \end{align*}
    Moreover, $\varepsilon_w(J_k/J_k^2)\varepsilon_v$ is a free $k$-module with basis the cosets of the arrows from $v$ to $w$.  
\end{Thm}
\begin{proof}
    We shall freely use the well-known fact~\cite{Webb,MagHomology} that if $I'$ is an ideal with $kQ/I'$ projective over $k$, then $I'$ is projective as a left $kQ$-module. Also, if $P$ is a projective $kQ$-module, then $P/I'P\cong kQ/I'\otimes_{kQ}P$ is a projective $kQ/I'$-module and hence projective over $k$.  In particular,  if $I_1,I_2$ are ideals with $kQ/I_j$ projective over $k$, for $i=1,2$, then $I_2/I_1I_2$ is projective over $k$ and $0\to I_2/I_1I_2\to kQ/I_1I_2\to kQ/I_2\to 0$ is a split exact sequence $k$-modules, and so $kQ/I_1I_2\cong  I_2/I_1I_2\oplus kQ/I_2$ is projective over $k$.
     Finally, note that if $I'$ is an ideal of a ring that is projective as a left module, then $I'e$ is projective for any idempotent $e$.  

    Since $kQ/I$ and $kQ/J_k\cong k^{Q_0}$ are $k$-projective, we then have a Gruenberg projective resolution~\cite{Webb,MagHomology} of  $(KQ/J_k)\varepsilon_v$ over $kQ/I$ of the form
    \begin{equation*}\label{eq:Gruenberg}
       \cdots\to  (IJ_k/I^2J_k)\varepsilon_v\to (I/I^2)\varepsilon_v\to (J_k/IJ_k)\varepsilon_v \to (kQ/I)\varepsilon_v\to (kQ/J_k)\varepsilon_v\to 0
    \end{equation*}
    where the even terms are given by $P_{2n} = (I^n/I^{n+1})\varepsilon_v$ and the odd terms by $P_{2n+1} = (I^nJ_k/I^{n+1}J_k)\varepsilon_v$.
    Deleting $(kQ/J_k)\varepsilon_v$ and tensoring over $kQ/I$ with $\varepsilon_w(kQ/J_k)$ then yields the chain complex
    \[\cdots\to \varepsilon_w(IJ_k/J_kIJ_k)\varepsilon_v\to \varepsilon_w (I/J_kI)\varepsilon_v\to \varepsilon_w(J_k/J_k^2)\varepsilon_v\to \varepsilon_w(KQ/J_k)\varepsilon_v\]
    where the even terms are of the form $C_{2n} = \varepsilon_w(I^n/J_kI^n)\varepsilon_v$ and the odd terms are of the form $C_{2n+1} = \varepsilon_w(I^nJ_k/J_kI^nJ_k)\varepsilon_v$
    using that $I\subseteq J_k$.  Computing the homology of this chain complex gives the result, where for $n=1,2$, we can simplify to
    \begin{align*}
        \Tor_1^{kQ/I}(\varepsilon_w (kQ/J_k),(kQ/J_k)\varepsilon_v) &= \varepsilon_w(J_k/(I+J_k^2))\varepsilon_v=\varepsilon_w(J_k/J_k^2)\varepsilon_v\\
        \Tor_2^{kQ/I}(\varepsilon_w(kQ/J_k),(kQ/J_k)\varepsilon_v) &= \varepsilon_w\left(\frac{I\cap J_k^2}{IJ_k+J_kI}\right)\varepsilon_v = \varepsilon_w\left(\frac{I}{IJ_k+J_kI}\right)\varepsilon_v
    \end{align*}
    using twice that $I\subseteq J_k^2$.  As the final statement is immediate from the definitions, this completes the proof.
\end{proof}

Next we turn to the question of when an algebra is a path algebra modulo an admissible ideal over a principal ideal domain.  The following is a minor generalization of~\cite[Lemma~1.2.8]{benson}; we include a proof for completeness.

\begin{Lemma}\label{l:sur.rad2}
Let $R$ be a ring and $J$ a nilpotent ideal of $R$.  Let $R'$ be a subring such that $R'+J^2=R$.  Then $R'=R$.
\end{Lemma}
\begin{proof}
We prove by induction that $R'+J^n=R$ for all $n\geq 2$.  It will then follow from nilpotency of $J$ that $R'=R$.  By assumption, $R'+J^2=R$.  Assume that $R'+J^n=R$ with $n\geq 2$.  Let $r\in R$.  Then $r=r'+z$ with $r'\in R'$ and $z\in J^n$, and so it suffices to show that $z\in R'+J^{n+1}$, i.e., that $J^n\subseteq R'+J^{n+1}$.  Let $x\in J^{n-1}$ and $y\in J$.  Then by the inductive assumption, we can find $x'\in R'$ with $x-x'\in J^n$ and by hypothesis we can find $y'\in R'$ with $y-y'\in J^2$, whence $y'=y-(y-y')\in J$.  Then $xy = x(y-y')+ (x-x')y'+x'y'\in R'+J^{n+1}$.  This completes the proof.
\end{proof}

If $A=kQ/I$ with $I$ an admissible ideal, then putting $\mathcal J=J_k/I$, we have that $\mathcal J$ is a nilpotent ideal (as it contains a power of $J_k$), $A/\mathcal J\cong k^{Q_0}$ and $\mathcal J/\mathcal J^2\cong J_k/J_k^2$ is a finitely generated free $k$-module.  We establish a converse when $k$ is a principal ideal domain, generalizing a result of Gabriel~\cite{assem,benson}.  
\begin{Thm}\label{t:gabriel}
    Let $k$ be a principal ideal domain and $A$ a $k$-algebra such that:
    \begin{enumerate}
        \item $A$ is free over $k$;
        \item there is a nilpotent ideal $\mathcal J$ of $A$ such that 
        \begin{enumerate}
            \item $A/\mathcal J\cong k^X$ for some set $X$;
            \item $\mathcal J/\mathcal J^2$ is a finitely generated free $k$-module. 
        \end{enumerate}
    \end{enumerate}
   Then $A\cong kQ/I$ for a quiver $Q$ and an admissible ideal $I$ where  $Q_0=X$ and $Q_1$ is in bijection with a basis for $\mathcal J/\mathcal J^2$. 

   More precisely, suppose that $\{e_x\mid x\in X\}$ is a complete set of orthogonal idempotents lifting $\{\delta_x\mid x\in X\}$ and $B_{yx}\subseteq e_y\mathcal Je_x$ maps to a basis of the finitely generated free $k$-module $e_y(\mathcal J/\mathcal J^2)e_x$, for $x,y\in X$.   Define a quiver $Q$ by  $Q_0=X$, $Q_1(x,y) =B_{yx}$.  Then $\psi\colon kQ\to A$ given by $\psi(\varepsilon_x) = e_x$ and $\psi(b) =b$ for $b\in B_{yx}$ is surjective with kernel an admissible ideal.  
\end{Thm}
\begin{proof}
    First note that
    \begin{equation}\label{eq:peirce1}
        \mathcal J/\mathcal J^2 = \bigoplus_{x,y\in X}e_y(\mathcal J/\mathcal J^2)e_x
    \end{equation}
    and hence each $e_y(\mathcal J/\mathcal J^2)e_x$ is a finitely generated free $k$-module, and so we can find $B_{yx}$ as in the theorem statement.    

    It is straightforward to show that $\psi$ can be extended to paths in a unique way making it a homomorphism.
   To show surjectivity of $\psi$ it suffices to verify $\psi(kQ)+\mathcal J^2=A$ by Lemma~\ref{l:sur.rad2}.  Since the $\delta_x$ form a basis for $k^X$, we conclude that $A=\bigoplus_{x\in X}ke_X+\mathcal J\subseteq \psi(kQ)+ \mathcal J$.  Therefore, we just need to check that $\mathcal J\subseteq \psi(kQ)+\mathcal J^2$. But this follows as \[B=\bigcup_{x,y\in X}B_{yx}\] projects to  a basis for $\mathcal J/\mathcal J^2$ by \eqref{eq:peirce1}, and $B\subseteq \psi(kQ)$.  

    We check that $I=\ker \psi$ is admissible.  First note that $kQ/I\cong A$ is a free $k$-module.  Also, we have $\psi(J)\subseteq \mathcal J$ by definition of $\psi$.  Therefore, if $\mathcal J^n=0$ with $n\geq 2$, then $\psi(J^n)\subseteq \mathcal J^n=0$.  Thus $J^n\subseteq I$.


    Suppose now that $a\in I=\ker \psi$.  Write \[a=\sum_{x\in X} c_x\varepsilon_x+\sum_{b\in B}d_bb+s\] with $s\in J^2$.  Then 
    \begin{align*}
    0=\psi(a) = \sum_{X\in \Lambda(B)}c_Xe_X +\sum_{b\in B}d_bb+\psi(s)
    \end{align*}
    and note that $\psi(s)\in \mathcal J^2$.  Therefore, $0=\psi(a)+\mathcal J = \sum_{x\in X}c_xe_X+\mathcal J$, and hence $c_X=0$ for all $X\in \Lambda(B)$ since the $e_X+\mathcal J$ are linearly independent.
    We now obtain 
    \begin{align*}
    0=\psi(a)+\mathcal J^2 = \sum_{b\in B}d_bb +\mathcal J^2.
    \end{align*}
    But $B$ projects to a basis for $\mathcal J/\mathcal J^2$.
    Therefore, the $d_b$ are all $0$, and so $a=s\in J^2$.   Thus $I\subseteq J^2$, completing the proof that $I$ is admissible.
\end{proof}

\begin{Rmk}
    The quiver $Q$ in Theorem~\ref{t:gabriel} is determined up to isomorphism by $A$.  Since $k$ is an integral domain, $\mathcal J$ is the unique maximal nilpotent ideal of $A$ and $A/\mathcal J\cong k^X$ has a unique set of primitive idempotents (necessarily, in bijection with $X$). This set determines $|Q_0|$ since for $I$ admissible and $k$ an integral domain, $J/I$ is the unique maximal nilpotent ideal of $kQ/I$ and $kQ/J\cong k^{Q_0}$.  Then if $e,f$ are two of these primitive idempotents, the cardinality of $Q_1(e,f)$ is the rank of $\Tor_1^A(f(A/\mathcal J),(A/\mathcal J)e)$ by Theorem~\ref{homological.stuff}.  
\end{Rmk}


\section{The integral quiver presentation}
Fix a finite band $B$.
If $X,Y\in \Lambda(B)$ with $X<Y$, we define  $\Gamma_{\mathscr R}(Y,X)$ be the simple graph with vertex set $L_X=\sigma\inv(X)\cap Bf_X$ and with the following two sorts of edges.  If $a\in L_X$, then there is an edge between $a$ and $ba$ for any $b\in B_{\geq Y}$ such that $ba\neq a$. We also place an edge between $a,c\in L_X$ if there is $b\geq_{\mathscr R} a,c$ with $\sigma(b)\ngeq Y$. In particular, $a$ and $f_Ya$ are connected by an edge, if they are not equal.      The graph  $\Gamma_{\mathscr L}(X,Y)$ for $X>Y$, is defined dually.  It has vertex set $R_Y=\sigma\inv(Y)\cap f_YB$.   If $a\in R_Y$, then there is an edge between $a$ and $ab$ for any $b\in B_{\geq X}$ such that $ab\neq a$. We also place an edge between $a,c\in R_Y$ if there is $b\geq_{\mathscr L} a,c$ with $\sigma(b)\ngeq X$.

\begin{Prop}\label{p:lrb.case}
    If $B$ is a left regular band, then each connected component of $\Gamma_{\mathscr R}(Y,X)$ meets the subgraph $\Gamma(Y,X)$ with vertex set $f_YBf_X\cap \sigma\inv(X)$ and adjacency given by having a common upper bound in $B^{<_\mathscr R f_Y}$ in a unique connected component.  In particular, the two graphs have the same number of connected components.
\end{Prop}
\begin{proof}
Clearly $\Gamma(Y,X)$ is a subgraph of $\Gamma_{\mathscr R}(Y,X)$ meeting every connected component since $b$ is adjacent to $f_Yb$ for $b\notin f_YB$.  It then suffices to show that $\Gamma(Y,X)$ is a retract of $\Gamma_{\mathscr R}(Y,X)$ by a simplicial map $\rho$, namely $\rho(b)=f_Xb$.  Indeed, if $b\geq_{\mathscr R} a,c$ with   $\sigma(b)\ngeq Y$, then $f_Yb\in B^{<_\mathscr R f_Y}$ and $f_Ybf_Ya = f_Yba=f_Ya$ and $f_Ybf_Yc= f_Ybc=f_Yc$ because $B$ is a left regular band.  Thus $f_Ya,f_Yc$ are adjacent (if not equal) in $\Gamma(Y,X)$.  If $b\in B_{\geq Y}$ and $a\in L_X$, then $f_Yb=f_Y$ since we are in a left regular band, and so $f_Yba=f_Ya$.    This completes the proof.
\end{proof}

Define a quiver $Q(B)$ as follows:
\begin{align*}
    Q(B)_0 &= \Lambda(B),\\
    |Q(B)_1(X,Y)| &= \begin{cases}|\pi_0(\Gamma_{\mathscr R}(Y, X))|-1, & \text{if}\ X<Y,\\ |\pi_0(\Gamma_{\mathscr L}(X, Y))|-1, & \text{if}\ X>Y,
    \\0, & \text{else.}\end{cases}
\end{align*}
  Note that if $B$ is a left regular band, then $\Gamma_{\mathscr L}(X,Y)$ is automatically connected, and so there are no arrows of the second type. Indeed, if $a,b\in \sigma\inv (Y)\cap f_YB$, then $a,b\leq_{\mathscr L}a$.

It was shown by 
Margolis and the author~\cite[Corollary~6.7]{DO}, generalizing earlier work of Saliola~\cite{Saliola} for left regular bands,
that the quiver of $kB$ is $Q(B)$ for any field $k$.  Our goal is to show that $\mathbb ZB\cong \mathbb ZQ(B)/I$ where $I$  is an admissible ideal. 
It will then follow that $kB$ has not only the same quiver for each field, but also the same presentation as a bound quiver.  

The following technical lemma will be used later. First, observe that $\mathcal J\mathbb ZL_X$ is spanned by all differences $(a-b)c=ac-bc$ with $\sigma(a)=\sigma(b)\geq X$ and $c\in L_X$ where $\mathcal J=\ker\sigma$.  It follows that it is, in fact, spanned by all differences $a-b$ such that $a,b\in L_X$, and hence has $\mathbb Z$-basis the elements $a-f_Yf_X$ with $a\in L_X\setminus \{f_Yf_X\}$.   Note that $\delta_{B_{\geq Y}}$ is the character of the $\mathbb ZB$-module $\mathbb Z^{\Lambda(B)}\delta_X=\mathbb Z\delta_X$.  Hence we can make $\til H_0(\Gamma_{\mathscr R}(Y,X))$ a $\mathbb ZB$-module by putting $b\cdot \ov z = \delta_{B_{\geq X}}(b)\ov z$.   

\begin{Lemma}\label{l:homology}
   There is a $\mathbb ZB$-module homomorphism \[\beta\colon \mathcal J\mathbb ZL_X\to \til H_0(\Gamma_{\mathscr R}(Y,X))\]given on the basis by $\beta(a-f_Yf_X) = \ov a-\ov{f_Yf_X}$ for $a\in L_X\setminus\{f_Yf_X\}$.  This induces a $\mathbb ZB$-linear map  $\mathcal J\mathbb ZL_X/\mathcal J^2\mathbb ZL_X\to \til H_0(\Gamma_{\mathscr R}(Y,X))$.
\end{Lemma}
\begin{proof}
    Let us put $\til H = \til H_0(\Gamma_{\mathscr R}(Y,X))$ for convenience.  If $\sigma(b)\ngeq X$, then $b$ annihilates both $\mathcal J\mathbb ZL_X$ and $\til H$.  Next suppose that $X\leq \sigma(b)$ and $\sigma(b)\ngeq Y$.  Then $b$ annihilates $\til H$.  We must show that if $a\in L_X$, then $ba-bf_Yf_X = ba-f_Yf_X - (bf_Yf_X -f_Yf_X) \in \ker \beta$.  But $b\geq_{\mathscr R} ba,bf_Yf_X$, and hence $\ov{ba}=\ov{bf_Yf_X}$.  It now follows that $\beta(ba-bf_Yf_X) = \ov{ba}-\ov{f_Yf_X}-(\ov{bf_Yf_X}-\ov{f_Yf_X})=0$.  Finally, consider $b\in B_{\geq Y}$.  If $a\in L_X$, then by construction of $\Gamma_{\mathscr R}(Y,X)$, we have that $a$ is adjacent or equal to $ba$ and $f_Yf_X$ is adjacent or equal to $bf_Yf_X$.  Therefore, $\beta(b(a-f_Yf_X)) = \ov {ba}-\ov{f_Yf_X}-(\ov{bf_Yf_X} - \ov{f_Yf_X}) = \ov a-\ov{f_Yf_X} = b\beta(a-f_Yf_X)$.

    The final statement follows because $\mathcal J$ annihilates $\til H$.
\end{proof}

Of course, a dual to Lemma~\ref{l:homology} applies to $R_Y=\sigma\inv(Y)\cap f_YB$ and $\Gamma_{\mathscr R}(X,Y)$ when $X>Y$.

Next, we analyze $e_Y(\mathcal J/\mathcal J^2)e_X$ in order to show that it is free abelian.

\begin{Lemma}\label{l:put.in.idems}
Let $\sigma(a)=\sigma(b)$ and $X,Y\in \Lambda(B)$.  Then $e_Y(a-b)e_X+\mathcal J^2 = e_Yf_Y(a-b)f_Xe_X+\mathcal J^2$.
\end{Lemma}
\begin{proof}
First we observe that 
\begin{equation}\label{eq:idems.1}
e_Y(a-b)e_X - e_Y(a-b)f_Xe_X = e_Y(a-b)(1-f_X)e_X\in \mathcal J^2,
\end{equation}
as $\tau((1-f_X)e_X) = (1-\delta_{\Lambda(B)^{\leq X}})\delta_X =0$, and $\sigma(a-b) =0$.  A dual argument shows that 
\begin{equation}\label{eq:idems.2}
     e_Y(a-b)f_Xe_X - e_Yf_Y(a-b)f_Xe_X\in \mathcal J^2.
\end{equation}
Putting together \eqref{eq:idems.1} and \eqref{eq:idems.2} yields the desired conclusion.
\end{proof}

Let us first deal with the case $X,Y$ are incomparable.

\begin{Lemma}\label{l:incomparable}
If $X,Y\in \Lambda(B)$ are incomparable,  then $e_Y(\mathcal J/\mathcal J^2)e_X =0$.    
\end{Lemma}
\begin{proof}
Note that $e_Y\mathcal Je_X$ is spanned by all $e_Y(a-b)e_X$ with $\sigma(a)=\sigma(b)$, and so we must show such elements belong to $\mathcal J^2$.  By Lemma~\ref{l:put.in.idems} it suffices to assume that $a,b\in f_YBf_X$.  Then $e_Y(a-b)e_X = e_Yaae_X-e_Ybbe_X\in\mathcal J^2$ since $\sigma(a)\leq X\wedge Y<X,Y$, and so $\tau(e_Ya)=\delta_Y\delta_{\Lambda(B)^{\leq \sigma(a)}}=0=\tau(ae_X)$ and, similarly,  $\tau(e_Yb)=0=\tau(be_X)$.  
\end{proof}

Next we consider with the case $X=Y$.

\begin{Lemma}\label{l:equal}
Let $X\in \Lambda(B)$.  Then $e_X(\mathcal J/\mathcal J^2)e_X =0$.    
\end{Lemma}
\begin{proof}
    It suffices to show that if $\sigma(a)=\sigma(b)$, then $e_X(a-b)e_X\in \mathcal J^2$.  If $\sigma(a)=\sigma(b)\ngeq X$, then $e_X(a-b)e_X = e_Xaae_X-e_Xbbe_X\in \mathcal J^2$, as $\tau(e_Xa)=0=\tau(ae_X)$ and $\tau(e_Xb)=0=\tau(be_X)$.  If $\sigma(a)=\sigma(b)\geq X$, then $f_Xaf_X=f_X = f_Xbf_X$ by swallowing, and so $f_X(a-b)f_X=0$.  Now, by Lemma~\ref{l:put.in.idems}, we have $e_X(a-b)e_X+\mathcal J^2 = e_Xf_X(a-b)f_Xe_X+\mathcal J^2=\mathcal J^2$, as required.
\end{proof}

The most technical case is when $X,Y$ are comparable.  

\begin{Lemma}\label{l:spanning.set}
    Let $X<Y$ be in $\Lambda(B)$.   Fix a set $T^r_{Y,X}\subseteq f_YBf_X$ of representatives of the connected components of $\Gamma_{\mathscr R}(Y,X)$ that contains $f_Yf_X$.  Then the elements of the form $e_Y(t-f_Yf_X)e_X+\mathcal J^2$ with $t\in T^r_{Y,X}\setminus \{f_Yf_X\}$ form a basis for $e_Y(\mathcal J/\mathcal J^2)e_X$ as a free abelian group.  

    Dually, if $X>Y$ in $\Lambda(B)$, and $T^\ell_{X,Y}\subseteq f_YBf_X$ is a set of representatives of the connected components of $\Gamma_{\mathscr L}(X,Y)$ that contains $f_Yf_X$, then the elements of the form $e_Y(t-f_Yf_X)e_X+\mathcal J^2$ with $t\in T^\ell_{X,Y}\setminus \{f_Yf_X\}$ form a basis for $e_Y(\mathcal J/\mathcal J^2)e_X$ as a free abelian group.
\end{Lemma}
\begin{proof}
    Notice that we can choose such a set of representatives $T^r_{Y,X}$ because there is an edge from $b$ to $f_Yb$ when $b\notin f_YBf_X$  (and dually for $T^\ell_{X,Y}$).
    We prove only the first statement as the proof of the second is dual.  

    Observe that $e_Y(\mathcal J/\mathcal J^2)e_X$ is generated by all elements of the form $e_Y(a-b)e_X+\mathcal J^2$ with $\sigma(a)=\sigma(b)$ and $a,b\in f_YBf_X$ by Lemma~\ref{l:put.in.idems}. Consequently, $\sigma(a)=\sigma(b)\leq X$.  Suppose that the inequality is strict. 
    Then $e_Y(a-b)e_X = e_Yaae_X-e_Ybbe_X\in \mathcal J^2$ as $\tau(e_Ya)=0=\tau(ae_X)$ and $\tau(e_Yb)=0=\tau(be_X)$   because $\sigma(a)=\sigma(b)<X<Y$.
    We conclude that $e_Y(\mathcal J/\mathcal J^2)e_X$ is generated by all elements of the form $e_Y(a-b)e_X+\mathcal J^2$ with $\sigma(a)=\sigma(b)=X$ and $a,b\in f_YBf_X$. 
    
    Next we claim that if $x$ and $x'$ are adjacent in $\Gamma_\mathscr R(Y,X)$, then we have $e_Y(x-x')e_X\in \mathcal J^2$. First suppose  there exists $z$ with $\sigma(z)\ngeq Y$ and with $zx=x$, $zx'=x'$.  
    Since $\sigma(x)=\sigma(x')$, it follows that $(x-x')e_X\in \mathcal J$. But $\tau(e_Yz) = \delta_Y\delta_{\Lambda(B)^{\leq \sigma(z)}}=0$, as $\sigma(z)\ngeq Y$.  Thus $e_Yz\in\mathcal J$.  Therefore, $e_Y(x-x')e_X= e_Yz(x-x')e_X\in \mathcal J^2$.  Next suppose that $\sigma(b)\geq Y$ and $a\in L_X$.  Then $e_Y(a-ba)e_X=e_Y(1-b)ae_X=e_Y(1-b)(1-b)ae_X\in \mathcal J^2$ because $\tau(e_Y(1-b))=0$ and $\tau((1-b)ae_X)=0$, as $\sigma(b)\geq Y>X$ and $\sigma(a)=X$.

    It now follows if we have $x_1,x_2,\ldots, x_n$ with $x_i,x_{i+1}$ adjacent in $\Gamma_\mathscr R(Y,X)$, then \[e_Y(x_1-x_n)e_X= \sum_{i=1}^{n-1}e_Y(x_i-x_{i+1})e_X\in \mathcal J^2.\]

    Let us write $\wh x$ for the element of $T^r_{Y,X}$ in the same component as $x$ for $x\in f_YBf_X\cap \sigma\inv(X)$.  
    By the above paragraph, if $\sigma(x)=\sigma(x')=X$ and $x,x'\in f_YBf_X$, then we have that
    \begin{align*}
        e_Y(x-x')+\mathcal J^2 &=  e_Y((x-\wh x)+(\wh x-\wh {x'}) - (x'-\wh{x'}))e_X+\mathcal J^2 \\
        &= e_Y(\wh x -\wh{x'})e_X + \mathcal J^2 \\ &=
        e_Y(\wh x-f_Yf_X)e_X - e_Y(\wh {x'}-f_Yf_X)e_X+\mathcal J^2 
    \end{align*}
    from which it follows that the elements $e_Y(t-f_Yf_X)e_X+\mathcal J^2$ with $t\in T^r_{Y,X}\setminus \{f_Yf_X\}$ generate $e_Y(\mathcal J/\mathcal J^2)e_X$

    Let us now prove that $e_Y(\mathcal J/\mathcal J^2)e_X$ is free abelian with these elements as a basis.   By Proposition~\ref{p:linearindep.schutz}, we have a $\mathbb ZB$-module homomorphism $\alpha\colon \mathcal Je_X\to \mathcal J\mathbb ZL_X$ that takes $(a-b)e_X$ to $a-b$ if $a,b\in L_X$.   We can compose this with the $\mathbb ZB$-module homomorphism $\beta\colon \mathcal J\mathbb ZL_X\to \til H_0(\Gamma_{\mathscr R}(Y,X))$ from Lemma~\ref{l:homology}.  Moreover, $\mathcal J^2e_X$ is in the kernel of $\beta\alpha$ by Lemma~\ref{l:homology}.  Since $e_Y$ fixes $\til H_0(\Gamma_{\mathscr R}(Y,X))$, we have that $\beta\alpha(e_Yje_X) = \beta\alpha(je_X)$ for $j\in \mathcal J$.     
    The elements $\ov t-\ov{f_Yf_X}$ with $t\in T^r_{Y,X}\setminus \{f_Yf_X\}$ are linearly independent in $\til H_0(\Gamma_{\mathscr R}(Y,X))$.    But by construction \[\beta\alpha(e_Y(t-f_Yf_X)e_X)=\beta\alpha((t-f_Yf_X)e_X) = \beta(t-f_Yf_X)=\ov t-\ov{f_Yf_X}.\]  Since $\mathcal J^2e_X\subseteq \ker \beta\alpha$, we deduce that the elements $e_Y(t-f_Yf_X)e_X+\mathcal J^2$ of $e_Y(\mathcal J/\mathcal J^2)e_X$ are linearly independent over $\mathbb Z$.  This completes the proof.
    \end{proof}

The following is our main theorem, which leads to a uniform presentation of $kB$ via a bound quiver for all fields $k$.  Previously, it was known only that the quiver was field independent.  

\begin{Thm}\label{t.main}
    Let $B$ be a connected band.  Then there is an isomorphism $\mathbb ZB\cong \mathbb ZQ(B)/I$ for a certain admissible ideal $I$.  
\end{Thm}
\begin{proof}
    For the purposes of this proof we think of $Q(B)$ as follows:
    \begin{align*}
    Q(B)_0 &= \Lambda(B),\\
    Q(B)_1(X,Y) &= \begin{cases}T^r_{Y,X}\setminus \{f_Yf_X\}, & \text{if}\ X<Y,\\ T^\ell_{X,Y}\setminus \{f_Yf_X\}, & \text{if}\ X>Y,\\ 0, & \text{else}\end{cases}
\end{align*}
where $T^r_{Y,X}\subseteq f_YBf_X$ is a set of representatives of the connected components of the graph $\Gamma_{\mathscr R}(Y,X)$ that contains $f_Yf_X$ and $T^\ell_{X,Y}\subseteq f_YBf_X$ is a set of representatives of the connected components of the graph $\Gamma_{\mathscr L}(X,Y)$ that contains $f_Yf_X$.
Using Lemmas~\ref{l:incomparable} and~\ref{l:equal} and that the $e_X$ form a complete set of orthogonal idempotents, we have that 
    \begin{equation*}\label{eq:peirce}
        \mathcal J/\mathcal J^2 = \bigoplus_{X<Y}e_Y(\mathcal J/\mathcal J^2)e_X\oplus \bigoplus_{X>Y}e_Y(\mathcal J/\mathcal J^2)e_X.
    \end{equation*}
Each summand is finitely generated free abelian by Lemma~\ref{l:spanning.set}, whence $\mathcal J/\mathcal J^2$ is finitely generated free abelian.  The summands with $X<Y$ have bases   $\{e_Y(t-f_Yf_X)e_X\mid t\in T^r_{Y,X}\setminus\{f_Yf_X\}\}$ and those with $X>Y$ have   bases $\{e_Y(t-f_Yf_X)e_X\mid t\in T^{\ell}_{X,Y}\setminus\{f_Yf_X\}\}$.  Since $\mathbb ZB/\mathcal J\cong \mathbb Z^{\Lambda(B)}$, Theorem~\ref{t:gabriel} applies to give an isomorphism $\mathbb ZQ(B)/I\to \mathbb ZB$ with $I$ admissible, induced by $\psi\colon \mathbb ZQ(B)\to \mathbb ZB$ with $\psi(\varepsilon_X)=e_X$ and $\psi(t) = e_Y(t-f_Yf_X)e_X$ for $t\in T^r_{Y,X}\setminus \{f_Yf_X\}$ with $X<Y$ or $t\in T^\ell_{X,Y}\setminus \{f_Yf_X\}$ with $X>Y$.  This completes the proof.  
\end{proof}

\begin{Thm}\label{t:main2}
    Let $B$ be a connected band.  Then there is an admissible ideal $I$ in $\mathbb ZQ(B)$ such that, for any commutative ring $k$ (and, in particular, all fields), $k\otimes I$ is an admissible ideal of $kQ(B)$ and $kB\cong kQ(B)/(k\otimes I)$.   
\end{Thm}
\begin{proof}
    The existence of an admissible ideal $I$  with $\mathbb ZQ(B)/I\cong \mathbb ZB$ is from Theorem~\ref{t.main}.  Proposition~\ref{p:admiss.ov.z} tells us that $k\otimes I$ is an admissible ideal of $kQ(B)$ and $kQ(B)/(k\otimes I)\cong k\otimes \mathbb ZQ(B)/I\cong k\otimes \mathbb ZB\cong kB$, completing the proof.
\end{proof}

Theorem~\ref{t:main2} provides a uniform quiver presentation of $kB$ over all fields.

\begin{Cor}\label{cor:hereditary}
      Let $B$ be a connected band. Then the following are equivalent.
      \begin{enumerate}
          \item $Q(B)$ is acyclic and $\mathbb ZB\cong \mathbb ZQ(B)$.
          \item $kB$ is hereditary for some field $k$.
          \item $kB$ is hereditary for all fields $k$.
      \end{enumerate}
\end{Cor}
\begin{proof}
    Note that (1) implies (3) as $kB\cong k\otimes \mathbb ZB\cong k\otimes \mathbb ZQ(B)\cong kQ(B)$ is hereditary~\cite{assem,benson}.  Trivially, (3) implies (2).  Suppose that (2) holds and let $I$ be an admissible ideal such that $\mathbb ZB\cong \mathbb ZQ(B)/I$ as per Theorem~\ref{t:main2}.    Now $kB\cong kQ(B)/(k\otimes I)$ and $k\otimes I$ is admissible by Theorem~\ref{t:main2}.  A theorem of Gabriel~\cite{assem,benson}  then tells us that $Q(B)$ is acyclic and $k\otimes I=0$.  Since $I$ is free abelian, we deduce that $I=0$, and so $\mathbb ZB\cong \mathbb ZQ(B)$.
\end{proof}

Note that (1) holds if and only if $Q(B)$ is acyclic and the number of paths in $Q(B)$ is $|B|$. Also note that $\mathbb ZB\cong \mathbb ZQ(B)$ already implies that $Q(B)$ is acyclic by finiteness of $B$.

The equivalent conditions of Corollary~\ref{cor:hereditary} are satsified when $B$ is a left regular band with identity and the Hasse diagram of $B$ is a rooted tree; see~\cite{MSS}.

We end this section with an example.
\begin{Example}
 Let $B$ be the monoid with presentation \[\langle a,b\mid a^2=a,b^2=b, aba=a, bab=b\rangle.\]  It is straightforward to verify that $B$ is a $5$-element band with elements $\{1,a,b,ab,ba\}$.  The support semilattice is the two-element lattice $\wh 0,\wh 1$, where $a,b,ab,ba$ have support $\wh 0$ and $1$ has support $\wh 1$.  We take $f_{\wh 0}=a$ and $f_{\wh 1}=1$.  Then $\Gamma_{\mathscr R}(\wh 1,\wh 0)$ has vertices $a,ba$ and no edges, and $\Gamma_{\mathscr L}(\wh 1,\wh 0)$ has vertices $a,ab$ and no edges.  It follows that $Q(B)$ is as depicted in Figure~\ref{fig:quiver} where $\alpha$ corresponds to $ab$ and $\beta$ to $ba$ in the labeling in the proof of Theorem~\ref{t.main}.
 \begin{figure}
     \centering
    \begin{tikzpicture}[->,shorten >=1pt,%
auto,node distance=2cm,semithick,
inner sep=5pt,bend angle=45]
\node[state] (A) {$\wh 0$};
\node[state] (B) [right of=A] {$\wh 1$};
\path
      (A) edge [bend right] node [below]  {$\beta$} (B)
      (B) edge [bend right] node [above] {$\alpha$} (A);
\end{tikzpicture}
     \caption{The quiver of a $5$-element band}
     \label{fig:quiver}
 \end{figure}
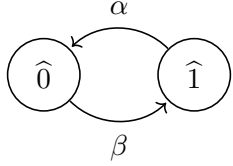
 We may take $e_{\wh 0} =f_{\wh 0}=a$ and $e_{\wh 1}=1-e_{\wh 0} = 1-a$.  Then we have
 \begin{align*}
 \psi(\varepsilon_{\wh 0}) = a,
 \psi(\varepsilon_{\wh 1}) = 1-a,
 \psi(\beta) = ba-a,
 \psi(\alpha) = ab-a. 
 \end{align*}
 Also notice that $\psi(\alpha\beta) = (ab-a)(ba-a) = 0$.  Thus $\alpha\beta\in I=\ker\psi$.  We claim that $I = (\alpha\beta)$.  Indeed, $\mathbb ZQ(B)/(\alpha\beta)$ is spanned by the cosets of the paths $\varepsilon_{\wh 0}$, $\varepsilon_{\wh 1}$, $\alpha$, $\beta$ and $\beta\alpha$.  These map under $\psi$ to $a$, $1-a$, $ab-a$, $ba-a$, $(ba-a)(ab-a) = b-ba-ab+a$, respectively.  These elements form a basis for $\mathbb ZB$, and so it follows that $\mathbb ZQ(B)/(\alpha\beta)$ maps isomorphically under the map induced by $\psi$ to $\mathbb ZB$.  This proves that $I = (\alpha\beta)$.  
\end{Example}

\section{CW left regular bands}
A left regular band $B$ is called a \emph{CW left regular band} if $B_{\geq X}$ equipped with $\leq_{\mathscr R}$ is the face poset of a regular CW complex for each $X\in \Lambda(B)$.  Examples include left regular bands associated to hyperplane arrangements, oriented matroids, CAT(0) cube complexes and complexes of oriented matroids (COMs).  A detailed discussion of CW left regular bands can be found in~\cite[Chapter~3]{ourmemoirs}.

It is shown in~\cite[Theorem~6.2]{ourmemoirs} that if $B$ is a connected CW left regular band, then $Q(B)$ is the Hasse diagram of $\Lambda(B)$ (with edges oriented upward) and that, for any field $k$, one has that $kB\cong kQ(B)/I$ where $I$ is generated by the sum of all paths of length $2$ in $Q(B)$.  We show in this section that $\mathbb ZB\cong \mathbb ZQ(B)/I$ where $I$ is generated by the sum of all paths of length $2$, and that this ideal is admissible.   This answers the question at the end of page~98 of~\cite{ourmemoirs}.

We assume that our idempotents $e_X$ are constructed as per~\cite[Theorem~4.2]{ourmemoirs}.
Let $\Sigma(B)$ be the unique (up to isomorphism) regular CW complex with face poset $B$.  We view the cellular chain complex with $\mathbb Z$-coefficients as a differential abelian group $C_\bullet(\Sigma(B))=\bigoplus_{q\geq 0} C_q(\Sigma(B))$.  Note that $B$ acts on $\Sigma(B)$ by regular cellular maps with the induced action on the face poset just the action via left multiplication. Moreover, the cells of $\Sigma(B)$ can be oriented so that the action of $B$ is orientation preserving. See Chapters~3 and ~5 of~\cite{ourmemoirs} for details. We write $x_*$  for the chain map on $C_\bullet(\Sigma(B))$ induced by left multiplication by $x\in B$.

There is then an isomorphism of $\mathbb ZB$-modules \[\Phi\colon C_\bullet(\Sigma(B))\to \mathbb ZB\] sending an oriented
cell $b\in B$ to $be_{\sigma(b)}$; see Page 97 of~\cite{ourmemoirs}.  We can define a differential $\partial$ on $\mathbb ZB$ by transporting the differential $d$ on the cellular chain complex, that is, $\partial = \Phi\circ d\circ\Phi\inv$.  We recall that $d(b) = \sum_{x\prec b}[b:x]x$ where $x\prec b$ means $b$ covers $x$ and $[b:x]=\pm 1$ is the incidence number.

We shall freely use in what follows that $B$ and $\Lambda(B)$ are graded posets and $\sigma\colon B\to \Lambda(B)$ preserves ranks, cf.~\cite[Prop.~3.8, Cor.~3.9]{ourmemoirs}.

\begin{Lemma}\label{l:opp.signs}
    Let $b\in B$ and suppose that $\sigma(b)$ covers $X$.  Then $b$ covers exactly two elements with support $X$ and $d(b)$ gives them opposite signs.  
\end{Lemma}
\begin{proof}
   Note that $b$ represents an edge of the CW poset $B_{\geq X}$, and so has two vertices, necessarily belonging to $X$ as this is the unique support covered by $Y$ in $B_{\geq X}$.  Thus $b$ covers two elements $x,x'$ with support $X$.  Notice that $x=xb=xx=xx'$, and so if $b$ has rank $q$, then the image of $b\in C_q(\Sigma(B))$ under $x_*$ is $0$. Also, if $y\prec b$ and $\sigma(y)\neq X$, then $xy$ has rank at most $q-2$. Thus $0=\partial x_\ast(b) = x_\ast \partial(b)=[b:x]x+[b:x']x$, and so $[b:x]=-[b:x']$.  
\end{proof}

We can now answer positively the question at the the end of page~98 of~\cite{ourmemoirs}.  

\begin{Thm}
    Let $B$ be a connected CW left regular band.  Let $Q$ be the Hasse diagram of $\Lambda(B)$ with the edges oriented upward.  Then $\mathbb ZB\cong \mathbb ZQ/I$ where $I$ is the admissible ideal generated by the sum of all paths of length $2$.  Hence, $k\otimes I$ is admissible and $kB\cong kB/(k\otimes I)$ for all commutative rings $k$.
\end{Thm}
\begin{proof}
    We observed already that $Q(B)=Q$.  If $X\prec Y$, then $\Gamma_{\mathscr R}(Y,X)$ has two connected components by Proposition~\ref{p:lrb.case}  since $\Gamma(Y,X)$ from that proposition consists of two isolated vertices, namely the two elements of $X$ covered by $f_Y$ (as per Lemma~\ref{l:opp.signs}), one of them being $f_Yf_X$.  Moreover, Proposition~\ref{p:lrb.case} says that these two elements are in different components of $\Gamma_{\mathscr R}(Y,X)$. Since $B$ is a left regular band, $\Gamma_{\mathscr L}(X,Y)$ is connected for all $X>Y$.  Define $\Psi\colon \mathbb ZQ\to \mathbb ZB$ by sending $\varepsilon_X$ to $e_X$ and the unique edge $X\to Y$ for $X\prec Y$ to the element $\partial(f_Ye_Y)e_X$.  From Lemma~\ref{l:opp.signs} and the third displayed equation on page 98 of~\cite{ourmemoirs}, one has $0\neq \partial(f_Ye_Y)e_X = e_Y(x-x')e_X$ where $x,x'$ are the two elements with support $X$ covered by $f_Y$.  Exactly one of these will be  $f_Yf_X$, and so this homomorphism differs from $\psi$ in the proof of Theorem~\ref{t.main} only up to the sign the image of each path is given.  It follows that $\Psi$ is onto and $\ker\Psi$ is admissible.  

    Let $r$ be the sum of all paths of length $2$.  
    It is shown in the proof of~\cite[Thm~6.1 and Thm~6.2]{ourmemoirs}, that $r\in \ker 1_{\mathbb Q}\otimes \Psi$, and hence, since $\mathbb ZQ$ is free abelian and $\mathbb Q$ is flat, $r\in \ker \Psi$.  It is also shown in~\cite[Thm~6.1 and Thm~6.2]{ourmemoirs} that $k\otimes \ker \Psi$ is generated by $r$ for any field $k$.

    Note that since $Q$ is finite and acyclic, there are only finitely many directed paths in $Q$.  Therefore, the additive group of $\mathbb ZQ$ is finitely generated free abelian.
    Let $I=(r)$.  Then we have $0\to I\to \ker \Psi\to \ker\Psi/I\to 0$ is exact and $\ker \Psi$ is a finitely generated free abelian group, as $\mathbb ZQ$ is.  Therefore, $\ker \Psi/I$ is a finitely generated abelian group.
    
    Tensoring over a field $k$ gives the sequence 
    \begin{equation*}\label{eq:cokernel}
        k\otimes I\to k\otimes \ker \Psi\to k\otimes \ker\Psi/I\to 0
    \end{equation*}
    is exact.  The image of $k\otimes I$ is the ideal generated by $(r)$ in $kQ$.  But, as we already have observed, this is $k\otimes \ker \Psi$.  Therefore, $k\otimes \ker \Psi/I=0$ for every field $k$.  Taking $k$ to be the $p$-element field for a prime $p$, we see that $p(\ker \Psi/I)=\ker \Psi/I$ for every prime $p$.  But then, $\ker \Psi/I$ is a finitely generated divisible abelian group, and hence $\ker \Psi/I=0$.  Thus $\ker \Psi=I$.   

    It now follows that $I$ is admissible and $\mathbb ZQ/I\cong \mathbb ZB$.  Hence, by Proposition~\ref{p:admiss.ov.z}, the final statement holds.
\end{proof}


\end{document}